\documentclass[11pt]{amsart}

\usepackage{amsmath, amsthm}
\usepackage{eucal}

\usepackage{amssymb}
\usepackage{amscd}
\usepackage{palatino}
\usepackage{latexsym}
\usepackage{epsfig}
\usepackage{graphicx}
\usepackage{amsfonts}
\usepackage{psfrag}
\usepackage{caption}

\usepackage{url}
\usepackage{cite}

\usepackage[margin=1in]{geometry}

\input xy
\xyoption{all}
\UseComputerModernTips

\numberwithin{equation}{section}
\numberwithin{figure}{section}

\newtheorem{theorem}{Theorem}[section]

\newtheorem{lemma}[theorem]{Lemma}
\newtheorem{proposition}[theorem]{Proposition}
\newtheorem{corollary}[theorem]{Corollary}

\theoremstyle{definition}
\newtheorem{definition}[theorem]{Definition}
\newtheorem{remark}[theorem]{Remark}
\newtheorem{example}[theorem]{Example}

\newcommand{\C}{{\mathbb{C}}}

\newcommand{\Z}{{\mathbb{Z}}}
\newcommand{\Q}{{\mathbb{Q}}}
\newcommand{\R}{{\mathbb{R}}}

\renewcommand{\P}{{\mathbb{P}}}

\renewcommand{\t}{\mathfrak{t}}
\newcommand{\su}{\mathfrak{su}}

\newcommand{\h}{\mathfrak{h}}

\newcommand{\into}{\hookrightarrow}

\renewcommand{\mod}{\mathbin{/\!/}}

\DeclareMathOperator{\Fred}{Fred}
\DeclareMathOperator{\Stab}{Stab}
\DeclareMathOperator{\Lie}{Lie}

\DeclareMathOperator{\codim}{codim}

\DeclareMathOperator{\Crit}{Crit}

\DeclareMathOperator{\pt}{pt}

\DeclareMathOperator{\id}{id}

\newcommand{\hsm}{{\hspace{1mm}}}

\newcommand{\okf}{\mathbb{K}_{\mathrm{orb}}}
\DeclareMathOperator{\Hess}{Hess}

\begin{document}

\title[Torsion in the full orbifold $K$-theory of abelian symplectic quotients]{Torsion in the full orbifold $K$-theory\\ of abelian symplectic quotients}

\author{Rebecca Goldin}
\address{Mathematical Sciences MS 3F2, George Mason University, 4400 University Drive, Fairfax, Virginia 22030,
USA}
\email{rgoldin@math.gmu.edu}
\urladdr{\url{http://math.gmu.edu/~rgoldin/}}

\author{Megumi Harada}
\address{Department of Mathematics and
Statistics\\ McMaster University\\ 1280 Main Street West\\ Hamilton, Ontario L8S4K1\\ Canada}
\email{Megumi.Harada@math.mcmaster.ca}
\urladdr{\url{http://www.math.mcmaster.ca/Megumi.Harada/}}
\thanks{The first author  was partially supported by NSF-DMS Grant \#0606869 and by a George Mason University Provost's Seed Grant. The second author was partially supported by an NSERC Discovery Grant,
an NSERC University Faculty Award, and an Ontario Ministry of Research
and Innovation Early Researcher Award.  The third author was partially supported by NSF-DMS Grant \#0835507 and and by a President's Council of Cornell Women Affinito-Stewart Grant.}

\author{Tara S. Holm}
\address{Department of Mathematics, Malott Hall, Cornell
  University, Ithaca, New York 14853-4201, USA}
\email{tsh@math.cornell.edu}
\urladdr{\url{http://www.math.cornell.edu/~tsh/}}


\date{\today}


\begin{abstract}
Let $(M,\omega,\Phi)$ be a Hamiltonian $T$-space and let $H\subseteq
T$ be a closed Lie subtorus. Under some technical hypotheses on 
the moment map $\Phi$, we prove that 
there is no additive torsion in
  the {\bf integral full orbifold $K$-theory} of the orbifold
  symplectic quotient $[M/\!/H]$. 
 Our main technical tool is an
  extension to the case of moment map level sets the well-known result
  that  many  components of the moment map of a Hamiltonian $T$-space $M$ are
  Morse-Bott functions on $M$. As first applications, we conclude that a large class
  of symplectic toric orbifolds, as well as certain $S^1$-quotients of GKM spaces, have integral full orbifold $K$-theory
  that is free of additive torsion.  
Finally, we
 introduce the notion of {\bf semilocally Delzant} which allows us
  to formulate sufficient conditions under which the hypotheses of the main
  theorem hold. We illustrate our results using low-rank coadjoint orbits of type $A$ and $B$. 
\end{abstract}

\maketitle

\setcounter{tocdepth}{1}
\tableofcontents

\section{Introduction}\label{sec:intro}

The main purpose of this manuscript is to show that the {\bf integral
  full orbifold $K$-theory} of several classes of orbifolds
${\mathfrak X}$ arising as abelian symplectic quotients are free of
additive torsion. An important subclass of symplectic quotients
to which our results apply are {\bf orbifold toric varieties}, of
which {\bf weighted projective spaces} are themselves a special
case. 

Orbifold toric varieties are global quotients of a manifold by a
torus action, and are therefore a natural starting point for a study
of orbifolds. Many conjectures on orbifolds and orbifold invariants in
active areas of research (algebraic geometry, equivariant topology,
the theory of mirror symmetry, to name a few) 
have been first tested in the realm of orbifold toric varieties. 
More specifically, there has been historically \cite{Amr94b, Amr94a,
  Amr97, Kaw73, NisYos97} and also quite recently a burst of interest
in weighted projective spaces (and their {\bf{integral}} invariants);
for more recent work, see for instance \cite{BahFraRay07, BoiManPer06,
  CoaCorLeeTse06, GueSak08, Hol07a, Tym08}.

We note that the application of our main result to orbifold toric
varieties is in the spirit of the previous work in the study of (both
ordinary and orbifold) topological invariants of weighted projective
spaces \cite{Kaw73, Amr94a, GolHarHolKim08}. Moreover, recent work of
Hua \cite{Hua09} uses algebro-geometric methods to show that
Grothendieck groups of a large class of toric Deligne-Mumford stacks
are free of additive torsion. This part of our results
(Theorem~\ref{theorem:intro-toric} below) can be viewed as a full
orbifold K-theory analogue of his results in the topological category,
proved via symplectic geometric methods. However, 
the scope of our results is more general. While the above-mentioned
work all deal with certain cases of orbifold toric varieties, the
techniques in this manuscript, which build upon the symplectic and
equivariant Morse theoretic methods developed in
\cite{GolHarHolKim08}, allow us to prove that the full orbifold
$K$-theory is free of additive torsion in more general settings. In
particular, we discuss non-toric examples in the later sections.

Let $T$ denote a compact connected
abelian Lie group, i.e., a torus. Suppose $(M, \omega, \Phi)$ is a
Hamiltonian $T$-space, with moment map \(\Phi: M \to \t^*\). 
Furthermore, let \(\beta: H \hookrightarrow T\) 
be a closed Lie subgroup (i.e. a
subtorus). Let \(\Phi_H: M \to \h^*\) be the induced $H$-moment map
obtained as the composition of \(\Phi: M \to \t^*\) with the linear
projection\footnote{By slight abuse of notation we use $\beta$ to also
  denote the linear map $\h \to \t$ obtained as the derivative of the
  inclusion map $\beta: H \into T$.} \(\beta^*: \t^*
\to \h^*.\) 
Suppose \(\eta \in \h^*\) is a regular value of $\Phi_H$,
and denote by
$Z:=\Phi_H^{-1}(\eta) \subseteq M$
 the corresponding
level set. Since $\eta$ is a regular value, $Z$ is a smooth
submanifold of $M$, and $H$ acts locally freely on $Z$. 
Let
\begin{equation}
  \label{eq:def-orbiquotient}
  {\mathfrak X} := [ M/\!/_{\eta} H ] = [ Z/H ]
\end{equation}
denote the quotient stack associated to the locally free $H$-action on
\(Z.\) This is an orbifold, also referred to as a {\bf Deligne-Mumford stack in the
differentiable category}. Let $\xi \in \t$ and recall that
$\Phi^\xi: = \langle \Phi,\xi\rangle:M\rightarrow \R$ denotes the corresponding component of the 
moment map.

The full orbifold $K$-theory ${\rm K_{orb}}({\mathfrak X})$ over $\Q$ was introduced by
Jarvis, Kaufman and Kimura in  \cite{JKK07}.  In the case that $\mathfrak X$ is formed as
an abelian quotient of a manifold $Z$ by a
locally free action of a torus $T$, the authors
of this manuscript
and Kimura gave an integral lift 
$\okf({\mathfrak X})$ in terms of the inertial $K$-theory $NK_T(Z)$ in \cite{GolHarHolKim08}.
This satisfies
$\okf({\mathfrak  X})\otimes \Q\cong {\rm K_{orb}}({\mathfrak X})$ as rings
\cite{BecUri07}.   Specifically, the full orbifold $K$-theory may be
described additively as a module over $K_T(\pt)$ by
$$
\okf({\mathfrak X}) = \bigoplus_{t\in T} K_T(Z^t).
$$
This differs from previous definitions of
``orbifold $K$-theory,'' e.g. that of
Adem and Ruan \cite{AdeRua03}. We refer the reader to the
introduction of \cite{GolHarHolKim08} for a more detailed discussion
of other notions of orbifold $K$-theory in the literature. 
In this manuscript, 
for $G$ a compact Lie group and $Y$ a $G$-space, we let $K_G(Y) = K^0_G(Y)$ denote the
Atiyah-Segal topological $G$-equivariant $K$-theory
\cite{Seg68}.  This is built from $G$-equivariant vector bundles 
when $Y$ is a compact $G$-space, and 
$G$-equivariant maps 
$[Y, \Fred(\mathcal{H}_G)]_G$ if $Y$ is noncompact (here $\mathcal{H}_G$ is a Hilbert 
space that contains infinitely many copies of 
every irreducible representation of $G$, see e.g. \cite{AtiSeg04}).
We may now state our main theorem about the structure of $\okf({\mathfrak X})$.

\begin{theorem}\label{thm:main}
Let $(M, \omega, \Phi)$ be a Hamiltonian $T$-space, $\beta: H \into T$
a connected subtorus with induced moment map $\Phi_H := \beta^* \circ
\Phi: M \to \h^*$. Suppose $\eta \in \h^*$ is a
regular value of $\Phi_H$, 
let $Z := \Phi_H^{-1}(\eta)$ denote its level set, and
${\mathfrak X}:=[Z/H]$ the associated quotient orbifold stack. 
Suppose 
there exists $\xi\in \t$ such that the following conditions hold:
\begin{enumerate}
\item[(1)] $H\subseteq \overline{\exp(t\xi)}$, the closure of the one-parameter subgroup 
generated by $\xi$ in $T$; 
\item[(2)] $f := \Phi^{\xi} \vert_Z$ is proper and bounded below;
\item[(3)] for each $t \in H$, $\pi_0(\Crit(f\vert_{Z^t}))$ 
  is finite;
\item[(4)] for each $t \in H$ and each connected component $C$ of
  $\Crit(f\vert_{Z^t})$, 
\begin{enumerate}
\item[(a)] $K_H^0(C)$ contains no additive
  torsion, and 
\item[(b)] $K_H^1(C)=0$.
\end{enumerate} 
\end{enumerate}
Then $\okf({\mathfrak X})$ contains no
additive torsion. 
\end{theorem}

A direct consequence is that when the components of the critical set are isolated $H$-orbits, $\okf({\mathfrak X})$ contains no additive torsion (see Corollary~\ref{cor:iso}). We use this corollary to prove the following. 

\begin{theorem}\label{theorem:intro-toric}
Let  ${\mathfrak X}$ be a symplectic toric orbifold obtained as a 
symplectic quotient of a linear $H$-action on a complex affine space,
where $H$ is a connected compact torus. Then $\okf({\mathfrak X})$ is
free of additive torsion. 
\end{theorem}

As mentioned above, this corollary is similar in spirit to Kawasaki's result that the
integral cohomology of (the underlying topological spaces of) weighted
projective spaces are free of additive torsion. Kawasaki showed in
\cite{Kaw73} that the integral cohomology groups of (the coarse moduli
space of) a weighted projective space agree with those of a smooth
projective space, but the ring structure differs, with structure
constants that depend on the weights. Hence we expect that the
richness of the data in $\okf({\mathfrak X})$ for toric orbifolds is
also contained not in additive torsion but rather in the
multiplicative structure constants of the ring. We leave
this for future work.

The main theorem may be applied in situations other
than that of the Delzant construction of orbifold toric varieties; 
the content of the last two sections of this
manuscript is an exploration of other situations in equivariant symplectic
geometry in which the hypotheses also hold.
First, we observe in Section~\ref{sec:GKM} that
the hypotheses above on the relevant connected components $C$ hold for  $S^1$-symplectic
quotients of Hamiltonian $T$-spaces which are GKM, under a technical
condition on the choice of subgroup $S^1 \subseteq T$. Spaces with
$T$-action which satisfy the so-called ``GKM conditions,'' introduced in the influential work of Goresky-Kottwitz-Macpherson
\cite{GKM}, are extensively studied in equivariant algebraic 
geometry, symplectic geometry, and geometric representation theory,
and encompass a wide array of examples.  We use a corollary of the
main theorem to prove Theorem~\ref{theorem:GKM}, which states that the
full orbifold $K$-theory of the quotient of a GKM space by certain
circle subgroups is free of additive torsion.
Secondly, we explore in 
Section~\ref{sec:semilocal} how the essential properties of the
Delzant construction (which allow us to prove Theorem~\ref{theorem:intro-toric})
may in fact be placed in a more general framework of phenomena in
torus-equivariant symplectic geometry which may be informally described as
`taking place within a $T$-equivariant Darboux neighborhood
of an isolated $T$-fixed point.' The precise statements are given in detail in
Section~\ref{sec:semilocal}, where we introduce the notion of a closed
$H$-invariant subset of a Hamiltonian $T$-space being {\bf semilocally
  Delzant} (with respect to $H$), and make some initial remarks on
situations in which this notion applies. One class of spaces to which
our definitions apply are the generalized flag varieties $G/B$ and
$G/P$, which may be covered by Darboux neighborhoods given by the Weyl
translates of the open Bruhat cell. Furthermore, the natural
$T$-action on generalized flag varieties (that of the maximal torus
$T$ in $G$) is also well-known to be GKM. In both
Sections~\ref{sec:GKM} and~\ref{sec:semilocal}, we illustrate our
results using examples of this type.

\section[Morse-Bott theory on level sets of moment maps]{A local normal form and Morse-Bott theory on level sets of moment maps}\label{sec:normalform}

We begin with our main technical lemma (Lemma~\ref{lemma:MorseBott})
regarding the Morse-Bott theory of moment maps in equivariant
symplectic geometry. The techniques used to prove this result are
fairly standard in the field, but we have not seen this particular
formulation in the literature.  It is
well-known that components of moment maps $\Phi^{\xi}=\langle
\Phi,\xi\rangle:M\to \R$ are Morse-Bott functions on a Hamiltonian
$T$-space $M$, for any
$\xi\in\t$. In addition, these components induce Morse-Bott functions
on smooth symplectic quotients $M/\!/_{\eta} H$, where $H$ is a
closed Lie subgroup of $T$, and $\eta$ is a regular value of the
$H$-moment map $\Phi_H$ \cite{MarWei74}.

What seems heretofore unnoticed\footnote{However, a result of this nature appears to be
  implicit in the work of Lerman and Tolman on the classification of
  orbifold toric varieties \cite{LT97}, and even earlier
in work of Marsden and Weinstein \cite{MarWei74} and Atiyah \cite{Ati82}.}
is that a component $\Phi^{\xi}$ of the $T$-moment map, restricted to
the level set $\Phi_H^{-1}(\eta)$ itself, is also a Morse-Bott
function
 when $H$ is contained in the closure of the 
 subgroup generated by $\xi$. 
 This may be deduced from the following local normal form result of Hilgert, Neeb, and Plank \cite[Lemmata 2.1 and 2.2]{HilNeePla94}, which builds on work of Guillemin and Sternberg \cite[Chapter II]{GS84a}. 
 Note that generic $\xi$ satisfy this condition.

\begin{proposition}[Hilgert, Neeb, Plank]\label{prop:local-normal}
Let $(M, \omega, \Phi)$ be a Hamiltonian $T$-space with moment map
\(\Phi: M \to \t^*.\) Let \(p \in M.\) Then there exists a
$T$-invariant neighborhood \(\mathcal{U} \subseteq M\) of the orbit \(T \cdot p
\subseteq M,\) a subtorus $T_1 \subseteq T$ and a symplectic vector
space $V$ such that:
\begin{enumerate}
\item[1.] There is a decomposition $T = T_0 \times T_1$, where $T_0 =
  Stab(p)_0$ is the connected component of the identity in the
  stabilizer group of $p$ in $T$.
\item[2.] There is a $T$-equivariant symplectic open covering from an open subset
  $\mathcal{U}'$ of $T_1\times\t_1^*\times V$ onto $\mathcal{U}$,
  where the $T$-action on $T_1 \times \t_1^* \times V$ is given by
    \begin{equation}\label{eq:local-normal}
      \begin{split}
    (T_0 \times T_1) \times (T_1 \times\t_1^*\times V) & \to (T_1\times\t_1^*\times V)\\
    ((t_0,t_1),(g,\gamma,v)) & \mapsto (t_1\cdot g, \gamma, \rho(t_0) v),
    \end{split}
    \end{equation}
    where $\rho : T_0\to Sp(V)$ is a linear symplectic representation.
\item[3.] There exists a complex structure $I$ on $V$ such that $\langle
  v, w \rangle := \omega_V(Iv,w)$ defines a positive definite
   scalar product on
  $V$. Let \(V = \bigoplus_\alpha V_{\alpha}\) be the decomposition of
  $V$ into isotypic components corresponding to weights \(\alpha \in
  \t_0^*.\) With respect to these local coordinates, the moment map
  $\Phi'$ on $\mathcal{U}' \subseteq T_1 \times \t_1^* \times V$ is given by
\begin{equation}\label{eq:local-Phi}
\begin{split}
\Phi': \mathcal{U}' \subseteq T_1\times\t_1^*\times V & \to \t^*\cong \t_0^*\oplus \t_1^* \\
(g,\eta,v) & \mapsto  \Phi'(1,0,0) + \left(\frac{1}{2}\sum ||v_\alpha||^2\alpha,  \eta\right).
\end{split}
\end{equation}
\end{enumerate}
\end{proposition}

For any $\xi\in \mathfrak{t}$, let $T^\xi:= \overline{\exp(t\xi)}$ denote the closure of the one-parameter subgroup generated by $\xi\in\mathfrak{t}$. 
Using the notation set in the Introduction, we now have the following. 

\begin{lemma}\label{lemma:MorseBott}
Let $(M,\omega,\Phi)$ be a Hamiltonian $T$-space, and $H\subseteq T$ a subtorus.
Let  $Z:= \Phi_{H}^{-1}(\eta)$ be a level set of the moment map for the $H$ action at a regular value. The function
$$f:= \Phi^{\xi}|_Z: Z\rightarrow \R$$
is a Morse-Bott function on $Z$ for every $\xi\in \mathfrak{t}$
such that $H\subseteq T^\xi$.
\end{lemma}

\begin{proof}
We show that for any point \(p \in Z\) such that $df_p=0$,
\begin{enumerate}
\item[1.] the connected component of \(\Crit(f)\)  containing $p$ is a submanifold, where $\Crit(f)$ is the critical set of $f$, and
\item[2.] the Hessian of $f$ at $p$ is non-degenerate in the directions normal to the connected component of $\Crit(f)$ containing $p$.
\end{enumerate}
Since the conditions to be checked are purely local, we may argue separately for each point \(p\) in the critical set \(\Crit(f).\)

For the purposes of this argument,
we may assume without loss of generality that the $T$-equivariant symplectic open
cover $\mathcal{U}' \to \mathcal{U}$ of
Proposition~\ref{prop:local-normal} is in fact a $T$-equivariant
symplectomorphism. The only part of this
claim requiring justification
is the relationship, in general, between the moment maps $\Phi_1$ and
$\Phi_2$ associated to Hamiltonian $T$-spaces $(M_1, \omega_1,
\Phi_1)$ and $(M_2, \omega_2, \Phi_2)$ where there exists a
$T$-equivariant symplectic open cover \(\pi: M_1 \to M_2.\) Since by
assumption \(\pi_1^*\omega_2 = \omega_1\) and
\(\pi_*(\xi^{\sharp}_{M_1}) = \xi^{\sharp}_{M_2}\) for all \(\xi \in
\t\) where $\xi^{\sharp}_{M_i}$ denotes the infinitesmal vector fields
generated by $\xi$ on the $M_i$, it follows immediately from
Hamilton's equations that \(\pi^*\Phi_2\) may be chosen as a moment
map $\Phi_1$ for the $T$-action on $M_1$.
In particular, since $\pi$ is an open covering,
the local argument for $\Phi_1$ in a small enough neighborhood of a
point $p$ in $M_1$ translates directly to an analogous argument in
$M_2$ for $\Phi_2$. Therefore we henceforth assume
that~\eqref{eq:local-normal} and~\eqref{eq:local-Phi} locally
represent a neighborhood of $p$, and $\Phi$ near $p \in Z$, respectively.

We continue with a characterization of the critical points
$\Crit(f) \subseteq Z$.
Recall
$T^\xi := \overline{\exp(t\xi)}$.
Let $\Stab_{T^\xi}(p)$ denote the stabilizer group in $T^\xi$
of $p$ and 
$\codim(H, T^\xi)$
 the codimension of the
subgroup 
$H$
in $T^\xi$. Suppose \(p \in Z.\) We claim that $p \in \Crit(f)$
if and only if
\(\dim(\Stab_{T^\xi}(p)) = 
 \codim(H, T^\xi)).\) 
Note that $p\in Z$ immediately implies $\dim\Stab_{T^\xi}(p) \leq
\codim(H, T^\xi))$, 
 since $H$
acts locally freely on $Z$. By definition,
a point \(p \in Z\) is critical for $f$ if and only if  
\[
df_p(v) = \langle d\Phi_p(v), \xi \rangle = \omega_p(\xi^{\sharp}_p, v) = 0, \quad \forall v \in \mathsf{T}_pZ,
\]
where $\mathsf{T}_pZ$ denotes the tangent space at $p$ to $Z$.
Note also that the tangent space 
$$
\mathsf{T}_pZ = \mathsf{T}_p\Phi_H^{-1}(\eta) = (\mathsf{T}_p(H \cdot
p))^{\omega_p} \subseteq \mathsf{T}_pM.
$$
Thus \(p \in Z\) is critical for $f$ if and only if $$\xi^{\sharp}_p
\in ((\mathsf{T}_p(H \cdot p))^{\omega_p})^{\omega_p} = \mathsf{T}_p(H \cdot p).$$ Since
$\xi$ generates $T^\xi$, it follows that $p\in\Crit(f)$ if and only if
\begin{equation}\label{eq:Tp-in-Kp}
\mathsf{T}_p(T^\xi \cdot p) \subseteq \mathsf{T}_p(H \cdot p),
\end{equation}
Hence $\dim \Stab_{T^\xi}(p) \geq 
\codim(H, T^\xi)$. 
Thus $p \in Z$ is critical for $f$ if and only if $\dim
\Stab_{T^\xi}(p)= 
\codim(H, T^\xi)$.

The above argument shows that for any \(\xi \in \t\)
with $H\subseteq T^\xi$, 
the critical set \(\Crit(f)\) is precisely the union of sets of the form $Z^{(T')}$
for subtori $T'$ of $T^\xi$ such that $\dim(T') = 
\codim(H, T^\xi)$, where
\[
Z^{(T')} := \{ p \in Z:  \Stab_{T^\xi}(p) = T' \}
\]
consists of the points whose stabilizer group in $T^\xi$ is precisely
$T'$. 
Since 
$H$ 
acts locally freely on $Z$, a
subtorus $T'$ of $T^\xi$ as above has maximal dimension among subtori
of $T^\xi$ with nonempty $Z^{(T')}$.
Now let \(p \in \Crit(f)\).
Consider local coordinates near $p$ as
in~\eqref{eq:local-normal}, with $\Phi$ near $p$ described
by~\eqref{eq:local-Phi}.
Write \(\xi = \xi_0 + \xi_1\) for \(\xi_0 \in
\t_0, \xi_1 \in \t_1.\) We first determine the intersection of
$\Crit(f)$ with this coordinate chart, in terms of
these local coordinates. From the description of the $T=T_0 \times T_1$-action
in~\eqref{eq:local-normal}, and from the fact observed above that $p$ is in $\Crit(f)$
precisely when its
stabilizer subgroup is of maximal possible dimension, it follows that
$\Crit(f)$ is the set of points of the form
\(\{ (g, \gamma, v): v \in V_0 \}\) where $V_0$ is the subspace of $V$
on which $T_0$ acts trivially. In particular, 
$\Crit(f)$ is a submanifold of $Z$ near $p$.

Finally, we show that the Hessian of $f$ near
$p$ is nondegenerate on those tangent directions in $\mathsf{T}_pZ$ corresponding to tangent vectors of the form
\(\{ (0, 0, \sum_{\alpha \neq 0} v_\alpha) : v_\alpha \in V_\alpha,
\alpha \neq 0 \}\) in the chosen local coordinates.
Recall that for tangent vectors \(v, w \in \mathsf{T}_pZ,\) the Hessian
\(\Hess(f)_p(v,w)\) is computed by \({\mathcal
  L}_{\tilde{v}} {\mathcal L}_{\tilde{w}}(f)\) where
$\tilde{v}, \tilde{w}$ are arbitrary extensions of $v,w$ to vector
fields in a neighborhood of $p$ in $Z$ (and ${\mathcal L}_X$ denotes a
Lie derivative along a vector field $X$). In the local coordinates of
Proposition~\ref{prop:local-normal}, any two vectors of the form \(v =
(0,0, \sum_{\alpha \neq 0} v_\alpha), w = (0,0, \sum_{\alpha \neq
  0}w_\alpha)\) may be extended to a neighborhood as the constant vector field
\(\tilde{v} \equiv (0,0,\sum_{\alpha \neq 0} v_\alpha), \tilde{w}
\equiv (0,0, \sum_{\alpha \neq 0} w_{\alpha}).\) We then observe that the
description of $\Phi$ in~\eqref{eq:local-Phi} implies that for such a
$\tilde{w}$,
\[
{\mathcal L}_{\tilde{w}}(f) = df(\tilde{w}) = d(\Phi^{\xi_0} \vert_Z)(\tilde{w}),
\]
since $\tilde{w}$ contains no component in $\t_1^*$.  It then suffices to show that the Hessian of the
$\t_0^*$-component of $\Phi$ is nondegenerate in the directions
$\oplus_{\alpha \neq 0} V_\alpha$. From the local normal form of
$\Phi$ in~\eqref{eq:local-Phi}, this is just a standard quadratic
moment map for a linear symplectic action of a torus on a symplectic
vector space, so this non-degeneracy is classical (see
e.g.\ \cite{Ati82}).
\end{proof}

\section{The proof and a corollary of the main theorem}\label{sec:main}

We now prove the main theorem. The argument uses
equivariant Morse theory of the moment map, most of which is standard (see, for example, \cite{Kir84,TW99,Ler04,HarLan07}).
The novel feature here involves the use of a component of the moment map
on a level set
of a moment map for a partial torus action.  We use the same notation
as in the introduction.

\begin{proof} [\bf Proof of Theorem~\ref{thm:main}]

We first note that since the statement of the theorem involves only
the additive structure of $\okf$, we need only recall the definition
(and computation) of $\okf({\mathfrak X})$ as an additive group. In
\cite{GolHarHolKim08} (cf. also \cite{BecUri07}), the integral full orbifold $K$-theory of
orbifolds ${\mathfrak X}$ arising as abelian symplectic quotients (by a torus $H$) is
described via an isomorphism \cite[Remark 2.5]{GolHarHolKim08}
\[
\okf({\mathfrak X}) \cong NK_{H}(Z) := \bigoplus_{t \in H} K_{H}
  \left( Z^t \right)
\]
where the middle term is the $H$-equivariant integral {\bf inertial $K$-theory} of the manifold $Z:=(\Phi_{H})^{-1}(\eta)$, defined additively as the direct sum above.  We now show that the right-hand side is torsion free.

Note that  $Z^t = \left( \Phi_{H} \vert_{M^t} \right)^{-1}(\eta)$, so it is itself a level set for the $H$-moment map on $M^t$ for each $t\in T$. 
Suppose $\xi\in \mathfrak{t}$ satisfies the hypotheses of the theorem, and let $f=\Phi^\xi|_{Z}$.
Since $f$ is proper and bounded below, then clearly $f|_{Z^t}$ is also proper and bounded below.  It is now immediate that $\xi$ satisfies conditions (1)--(4) for the Hamiltonian $T$-space $M^t$.
Thus without loss of generality,
we need only check that $K_{H}\left( Z \right)$ is torsion-free; all
other cases follow similarly. 

By Lemma~\ref{lemma:MorseBott}, $f $ 
is a Morse-Bott function. 
Denote the connected components of $\Crit(f)$ by
$\{C_j\}_{j=1}^\ell$, where $\ell$ is finite by condition
~(3) 
 and assume
without loss of generality that $f(C_i)<f(C_j)$ if $i < j$. Because $f$ is bounded below and proper, all
components are closed and compact, and there exists a minimal component, which we denote $C_0$.

Assume $Z$ is nonempty. We 
build the equivariant $K$-theory of $Z$ inductively by studying the critical
sets, beginning with the base
case. By assumption, \(K^0_{H}(C_0) \) has no additive torsion and $K^1_{H}(C_0) =
0$. For small enough \(\varepsilon >0,\) consider the submanifolds
\[
Z^+_j = f^{-1}((-\infty, f(C_j) +
\varepsilon)),
\quad
Z^-_j = f^{-1}((-\infty, f(C_j) -
\varepsilon)),
\]
where $\varepsilon$ is chosen so that $C_j$ is the only critical component contained in $Z_j^+
\setminus Z_j^-$. Using the $2$-periodicity of (equivariant) $K$-theory, there is a periodic long
exact sequence
\begin{equation}
  \label{eq:LES}
 \begin{array}{c}
 \xymatrix{
 & K^0_{H}(Z_j^+) \ar[r] & K^0_{H}(Z_j^-)\ar[dr] & \\
  K^0_{H}(Z_j^+, Z_j^-)\ar[ur] & & &   K^1_{H}(Z_j^+,Z_j^-)\ar[dl] \\
 & K^1_{H}(Z_j^-) \ar[ul] & K^1_{H}(Z_j^+)\ar[l] &
 } \end{array}
\end{equation}
in equivariant $K$-theory for the pair $(Z_j^+, Z_j^-)$. 
Choose an $H$-invariant metric on $Z$, and identify $K^*_H(Z_j^+,
Z_j^-)$ with $K^*_H(D(\nu_j^-), S(\nu_j^-))$, where $D(\nu_j^-), S(\nu_j^-)$
are the disc and sphere bundles, respectively, of the negative normal bundle to $C_j$
with respect to $f$.  
The equivariant Thom isomorphism also says that
$K^*_H(D(\nu_j^-), S(\nu_j^-)) \cong K^*_H(C_j)$.
There is no degree shift
 since the (real) dimension of the negative normal bundle is even (as can be seen from
 Proposition~\ref{prop:local-normal}) and $K$-theory is
 $2$-periodic. By assumption, $K^1_H(C_j)=0$, and by the inductive
 assumption we have $K^1_H(Z_j^-)=0$. Hence we may immediately
 conclude from~\eqref{eq:LES} that $K^1_H(Z_j^+)=0$ and that there is
 a short exact sequence 
\begin{equation}
  \label{eq:SES}
  0 \to K^0_{H}(Z_j^+, Z_j^-) \to K^0_{H}(Z_j^+) \to
  K^0_{H}(Z_j^-) \to 0.
\end{equation}
By induction, \(K^0_{H}(Z_j^-)\) has no additive torsion, and by assumption,
\[
K^0_{H}(Z_j^+, Z_j^-) \cong K^0_H(D(\nu_j^-), S(\nu_j^-)) \cong
K^0_{H}(C_j)
\]
does not either. We conclude
that $K^0_H(Z_j^+)$ is also free of additive torsion. 
Hence by induction we conclude that $K^0_H(Z_\ell^+)$ is free of
additive torsion. Since $C_\ell$ is the maximal critical component,
there are no higher critical sets, so 
the negative gradient flow with respect to 
$f$ yields an $H$-equivariant deformation retraction from $Z$ to
$Z_\ell^+$. Hence $K_H(Z) \cong K_H(Z_\ell^+)$, and in particular we
may conclude that $K^0_H(Z)$ is free of additive torsion, as desired. 
\end{proof}

\begin{remark} 
In the course of the proof, we have also shown that $K^1_H(Z^t)=0$ for
all $t\in H$. In the inductive arguments given in Sections~\ref{sec:toric}
and~\ref{sec:GKM}, we will need this additional fact 
to obtain Theorems~\ref{thm:toric} and~\ref{theorem:GKM}. 
\end{remark}

We now turn to the first application of Theorem~\ref{thm:main}, the case when the critical set consists of isolated $H$-orbits.

\begin{corollary}\label{cor:iso}
 Let ${\mathfrak X}= [ Z/H ]$ be an orbifold constructed as in~\eqref{eq:def-orbiquotient}. As above, suppose
 that there exists $\xi\in \t$ such that 
\begin{itemize} 
\item $H\subseteq T^\xi$,
\item $f:=\Phi^\xi|_Z$ is proper and bounded below, and
\item for every $t\in H$, $Crit(f|_{Z^t})$ consists of 
  finitely many isolated $H$-orbits.
\end{itemize} 
Then  $\okf({\mathfrak X})$
contains no additive torsion. Furthermore,  $K^1_H(Z^t)=0$ for all $t\in H$.
 \end{corollary}

 \begin{proof}
It suffices to check that the hypotheses of Theorem~\ref{thm:main} are
satisfied, and it is evident that the only assumption needing comment
is 
(4).
Since each connected component $C$ is an isolated $H$-orbit, and by
assumption $H$ acts locally freely on $Z$, we have 
 \[
 K^0_{H}(C) \cong K^0_{H}(H \cdot p) \cong
 K^0_{H}(H/\Gamma),
 \]
where $p\in C$ and $\Gamma$ is the finite stabilizer subgroup $\Stab_T(p)$ in $H$. The $H$-equivariant $K$-theory of
 a homogeneous space is the representation ring of the stabilizer of the identity coset,
 \[
 K^0_{H}(H/\Gamma) \cong K^0_{\Gamma}(\pt) \cong R(\Gamma),
 \]
 which has no additive torsion. Moreover, $K^1_{H}(H/\Gamma) \cong K^1_{\Gamma}(\pt) = 0$.
Hence, assumptions (4a) and (4b) hold, and we may apply the Main
Theorem. The result follows. 
 \end{proof}

This corollary provides the starting point for inductive arguments
which show that the integral full orbifold $K$-theory of an abelian
symplectic quotient is torsion free.

\begin{remark}\label{remark:corollary}
It follows immediately from this proof
  that the integral full orbifold $K$-theory
  $\okf({\mathfrak X})$ of an orbifold ${\mathfrak X}= [ Z/H ]$
  satisfying the hypotheses of the Main Theorem
is additively the direct sum of representation rings  $R(\Gamma)$ for
 those
  subgroups $\Gamma$ of $H$ appearing as
  stabilizer groups in the level set of the moment map \(Z =
  \Phi_{H}^{-1}(\eta). \)  It would
  be interesting to compare this description via representation rings
  to the computation given in \cite{GolHarHolKim08} in terms of the
  Kirwan surjectivity theorem in full orbifold $K$-theory.
\end{remark}

\section{Symplectic toric orbifolds}\label{sec:toric}

We now provide a first application of the Main Theorem and its
corollary, namely:  for a large class of toric orbifolds,
the integral full orbifold $K$-theory contains no additive
torsion. 
In the case of weighted projective spaces 
similar results were obtained by Kawasaki in ordinary
integral cohomology in the 1970s \cite{Kaw73}, then in ordinary
$K$-theory (using results of \cite{Kaw73}) by Al Amrani in
\cite{Amr94a}. 
More recently, Zheng Hua  \cite{Hua09} has independently shown using
algebro-geometric methods that, when the generic point is stacky,
the Grothendieck group
$K_0({\mathfrak X}_\Sigma)$ of a smooth complete toric Deligne-Mumford
stack is a free $\Z$-module. Here, $\Sigma$ is a stacky fan as defined
in \cite{BCS05} and $K_0$ is the algebraic $K$-theory
defined via coherent sheaves. Since it is straightforward to see from
the definition (given below) of symplectic toric orbifolds ${\mathfrak
X}$ that the twisted sectors arising in the computation of the full
orbifold $K$-theory $\okf({\mathfrak X})$ are themselves stacks which
are symplectic toric orbifolds, the substantive statement (which is
the topological $K$-theory
analogue of Hua's result) is that each
twisted sector individually has $K$-theory free of additive torsion.
Hence Theorem~\ref{thm:toric} should be viewed as a straightforward integral full
orbifold $K$-theory
analogue, in the topological category and for symplectic toric
orbifolds ${\mathfrak X}$, of Hua's result  \cite{Hua09}. However, our methods
of proof, which use the equivariant Morse theory of symplectic
geometry developed in Sections 2 and 3, are significantly different
from those of  \cite{Hua09}.

We first establish notation for both the Delzant construction of toric varieties and the statement
of the theorem. In the smooth case, this construction may be found in \cite{Del88} (for an
accessible account, see \cite{CdS01}).  This construction is generalized to the orbifold case in
\cite{LT97}. Let $T^n = (S^1)^n$ be the standard compact $n$-torus, acting in the standard linear fashion on
$\C^n$ (via the embedding of $T^n$ into $U(n,\C)$ as diagonal matrices with unit complex
entries). This is a Hamiltonian $T^n$-action on $\C^n$ with respect to the standard K\"ahler
structure on $\C^n$. 

Let \(\Phi: \C^n \to (\t^n)^*\) denote a moment map for this action. For a
connected closed subtorus $\beta: H\hookrightarrow T^n$, let $\Phi_H:=\beta
\circ \Phi: \C^n\rightarrow \mathfrak{h}^*$ denote the induced moment
map. For a regular value $\eta\in
\mathfrak{h}^*$ of $\Phi_H$, let $Z:=\Phi_H^{-1}(\eta)$ be its level set.
By regularity of $\eta$, $H$ acts locally freely on
$Z$. The {\bf symplectic toric orbifold} specified by
$\beta: H \into T^n$ and $\eta$ is then
defined by
$$
\mathfrak{X} := \C^n/\!/_{\eta} H =[Z/H].
$$
The procedure just recounted is often called the {\bf Delzant construction} of the toric orbifold ${\mathfrak
X}$, although historically it was the underlying topological space of ${\mathfrak X}$ that was
studied, not the associated stack\footnote{Indeed, the underlying topological space \(Z/H\)
corresponding to the stack ${\mathfrak X}$ is often also called the {\bf symplectic quotient of
  $\C^n$ by $H$ at the value $\eta$}. In the current
  literature, there is an unfortunate ambiguity: 
  the ``symplectic quotient" may refer to the stack or the
  underlying topological space.
  }. 
Symplectic toric orbifolds were classified in \cite{LT97}; we
consider only those obtained by a quotient by a connected subtorus
$H$. We will call an element \(\xi \in \t\) of the Lie algebra {\bf
  generic} if its associated $1$-parameter subgroup $\exp(t\xi)$ in
$T$ is dense: in the notation of Section~\ref{sec:normalform}, $T^\xi
= T$. Note that if there exists any $\xi \in \t$ such that $\Phi^\xi
\vert_Z$ is proper and bounded below, then there also exists a generic
$\xi \in \t$ satisfying the same conditions.

\begin{theorem}\label{thm:toric}
Let ${\mathfrak X} = \C^n/\!/_{\eta} H$ be a symplectic toric
orbifold, where $\beta: H \into T$ a connected closed subtorus of $T$ and
$\eta \in \h^*$ a regular value.
Let $Z = \Phi^{-1}_H(\eta)$ denote the $\eta$-level set of $\Phi_H$.
Then $\okf({\mathfrak X})$
has no additive torsion. Furthermore, $K^1_H(Z^t)=0$ for all $t\in H$.
\end{theorem}

\begin{proof}
Since the original $T$-action on $\C^n$ is a standard linear action by diagonal matrices, for any
\(t \in H,\) the fixed point set $(\C^n)^t$ is a coordinate subspace,
i.e. $(\C^n)^t \cong \C^m \subset \C^n$, determined by the values of the $T$-weights on each coordinate line
$\{(0,0,\ldots, z_j, 0, \ldots, 0)\} \subseteq \C^n$. It is immediate that $(\C^n)^t$ is a linear
symplectic subspace of $\C^n$ and that the restriction \(\Phi_{H} \vert_{(\C^n)^t}: (\C^n)^t \to
\h^*\) is a moment map for this action. Thus \(Z^t\) is equal to\(\left( \Phi_{H} \vert_{(\C^n)^t}
\right)^{-1}(\eta),\) the level set of a moment map for a $H$-action on a
possibly-smaller-dimensional vector space.

Choose a generic \(\xi \in \t\) such that $\Phi^\xi \vert_Z$ is proper and
bounded below.  Such a $\xi$ exists because there are such components for
$\Phi:\C^n\to\mathfrak t^*$, and $Z$ is a $T$-invariant closed subset of $\C^n$.
Let \(f = \Phi^\xi \vert_Z.\) In order to apply Corollary~\ref{cor:iso},
we must check that for all 
\(t \in H,\) the critical set $\Crit(f \vert_{Z^t})$ consists of
finitely many isolated $H$-orbits. 
We first observe that since $(\C^n)^t \cong \C^m$ is itself a
symplectic linear space equipped with a linear $T$-action, it suffices
to prove this statement for the special case $t=\id$; the other cases
follow similarly.

Let $C$ be a connected component in $\Crit(f)$ and $p\in C$. Since $f$ is $T$-invariant, $H\cdot
p\subset C$. Since $C$ is compact and connected, it suffices to show that $C$ consists of one orbit
locally.
Recall from the proof of Lemma~\ref{lemma:MorseBott} that \(p \in \Crit(f)\) exactly if
\[\dim(\Stab_T(p)) = \codim(H) = n-k.\]
Thus \(\dim(\Stab_T(p)) = n-k\) exactly if $p = (z_1, z_2,\ldots, z_n) \in \C^n$ has precisely
$n-k$ coordinates equal to $0$, i.e.\ $p$ lies in a coordinate subspace of $\C^n$ isomorphic to
$\C^k$. Note that $H$ acts on $\C^n$ preserving this $\C^k$, and the regularity assumption on
$\eta$ implies that the restriction of $\Phi_{H}$ to $\C^k$ is $\Q$-linearly isomorphic to the
standard moment map for the standard $H$-action on $\C^k$ (up to a translation by a constant in
$\h^*$). In particular, this implies that the condition \(p \in Z:=\Phi_{H}^{-1}(\eta)\) for a
regular value $\eta$ uniquely determines the non-zero norms of the coordinates \(\|z_{i_1}\|^2,
\|z_{i_2}\|^2, \ldots, \|z_{i_k}\|^2.\) Therefore, the only nearby points $p'\in Z$ with
$\dim(\Stab_T(p')) = n-k$ are those in the $H$ orbit of $p$. We conclude that each connected
component $C$ of $\Crit(f)$ is a single $H$-orbit.  Moreover, there are only finitely many critical components 
because there are only finitely many $k$-dimensional coordinate
subspaces of $\C^n$. 

The same argument for each $(\C^n)^t$ and an application of
Corollary~\ref{cor:iso} 
completes the proof.
\end{proof}

\section{GKM spaces}\label{sec:GKM}

Let $(M,\omega,\Phi)$ be a compact Hamiltonian $T$-space. 
Suppose in addition that the $T$-fixed points are isolated, and that
the set of points with codimension-$1$ stabilizer 
\[
M_1 := \{x \in M \hsm \vert \hsm \codim(\Stab(x)) = 1\}
\]
has real dimension $\dim(M_1) \leq 2$. 
 When these conditions are satisfied, we say that 
$M$ is a {\bf GKM space} and that the
$T$-action on $M$
is {\bf GKM}.\footnote{There are many variants on the definition of
  GKM actions (see e.g. \cite{GZ01, GZ02, HarHol05, HHH05}).
  In particular, in less restrictive versions, the $T$-space $M$ need not be
  compact nor symplectic, nor even finite-dimensional.} 
It is also well-known in the theory of GKM spaces (in the context of
the study of Hamiltonian $T$-actions) that these conditions imply that
the {\bf equivariant $1$-skeleton of $M$}, i.e. the closure $\overline{M_1} = M_1 \cup M^T$, is a union of symplectic $2$-spheres
$S^2$. Moreover, each such $2$-sphere is itself a
Hamiltonian $T$-space; the $T$-action on $S^2$ is given by a
nontrivial character \(T \to S^1\) (equivalently, a nonzero weight \(\alpha \in
\t^*_{\Z}\)) where the $S^1$ acts on $S^2$ by rotation.
Here the weight $\alpha$ is obtained from the linear $T$-isotropy data at
either of the two
$T$-fixed points in $S^2$. (For details see e.g. the expository
article \cite{Tym05b}.) 

Hamiltonian $T$-spaces $(M,\omega,\Phi)$ (or
algebraic
varieties equipped with algebraic torus actions) for which the $T$-action is GKM have been
extensively studied in modern equivariant symplectic and algebraic
geometry, primarily due to the link provided by GKM theory between 
$T$-equivariant topology and the combinatorics of what is called the
{\bf moment graph} (or {\bf GKM graph}) of $M$. Many natural examples arise in the
realm of geometric representation theory and Schubert calculus, 
including generalized flag varieties $G/B$ and $G/P$ of
Kac-Moody groups $G$ (where $B$ is a Borel subgroup and, more
generally, $P$ a parabolic subgroup). Hence the orbifold invariants of
the orbifold symplectic quotients of GKM spaces is a natural area of study.

If the $T$-action is GKM, then for a large class of circle subgroups
of $T$, the associated orbifold symplectic quotients $M \mod_\eta S^1$
have no additive torsion in full integral orbifold $K$-theory, as we
now see.

\begin{theorem}\label{theorem:GKM}
Suppose that $(M,\omega,T,\Phi)$ is a compact Hamiltonian $T$-space,
and suppose further that the $T$-action is GKM. 
Suppose
that $\beta: S^1 \into T$ is a 
circle subgroup in $T$ such that $M^{S^1} = M^T$, and let
$\Phi_{S^1}:=\beta^* \circ \Phi: M\rightarrow \Lie(S^1)^*$ denote the induced moment map.  Let
$\eta\in \Lie(S^1)$ be a regular value of $\Phi_{S^1}$, and $\mathfrak{X}
=M/\!/_\eta S^1$ the orbifold symplectic quotient. Then $\okf(\mathfrak{X})$ is free of additive torsion.
\end{theorem}

\begin{proof} 
We show that the hypotheses of Corollary \ref{cor:iso} hold. Let $Z :=
\Phi_{S^1}^{-1}(\eta)$, choose $\xi
\in \t$ such that its $1$-parameter subgroup in $T$ is dense in $T$,
and let \(f := \Phi^\xi \vert_Z.\) 
Properness of $f$ is immediate since $M$ is compact. 
Hence it suffices to show that the critical sets $\Crit(f)$ and
$\Crit(f|_{Z^t})$ are isolated $S^1$-orbits.
Observe that 
when $M$ is a GKM space, $M^t$ is also a GKM space for any $t\in S^1$.
Hence it suffices to argue only for the case of $\Crit(f)$; the others follow
similarly.

By the argument given in the proof of Lemma~\ref{lemma:MorseBott}, $\Crit(f)$
consists precisely of those points \(p \in Z\) satisfying
\(\codim(\Stab_T(p)) = 1.\) In other words, \(\Crit(f) = Z \cap M_1.\) 
The closure $\overline{M_1}$ consists of a union of $2$-spheres, and the
$T$-action on each $S^2$ is specified by a non-zero weight \(\alpha
\in \t^*_{\Z}\) obtained from the $T$-isotropy decomposition at one of
the two fixed points of the $S^2$. By assumption on the circle
subgroup $S^1$, the kernel of the character $\phi_\alpha: T \to S^1$ specified by
$\alpha$ does not contain $S^1$. Therefore, $S^1$ acts nontrivially on each
$S^2 \subseteq\overline{M_1}$, implying $\Phi_{S^1}|_{S^2}$ is nontrivial,
and $\Phi_{S^1}^{-1}(\eta) \cap S^2$ consists of a single $S^1$-orbit.
(Note that $Z$ does not contain any $0$-dimensional orbits of $S^1$
since, by assumption on regularity of
$\eta$, $S^1$ acts locally freely on $Z$.) 

Thus the hypotheses of Corollary~\ref{cor:iso} are satisfied, so $\okf(\mathfrak{X})$ is additively torsion-free.

 \end{proof}

\begin{remark} 
We restrict to the case of compact symplectic manifolds in
this section for sake of brevity. 
However, the arguments given above could be altered to prove
analogous results in less-restrictive contexts of GKM theory (see e.g.
\cite{HarHol05}, \cite{HHH05}). 
\end{remark}

\begin{remark} It may be an interesting exercise to generalize
  Theorem~\ref{theorem:GKM} to symplectic quotients of GKM spaces by higher dimensional tori. 
 One approach would be to consider quotients of a {\bf $k$-independent
    GKM space} (cf. \cite{GZ01}) by a $(k-1)$-dimensional torus. 
\end{remark}

We now illustrate use of Theorem~\ref{theorem:GKM} 
for some coadjoint  orbits of low-rank Lie type. We will analyze examples derived from the
natural $G$-action on coadjoint orbits of $G$, but we must be careful
to avoid the possibility of non-effective actions (so the symplectic
quotient is an \textbf{effective} orbifold). Therefore, in
Examples~\ref{example:flagsC3}, \ref{example:B2}, and \ref{example:B3},
we use the quotient group $PG := G/Z(G)$ where $Z(G)$ denotes the
(finite) center of $G$; by slight abuse of notation, we also notate by
$T$ the image of the usual maximal torus under the quotient $G \to
PG$.

\begin{example}\label{example:flagsC3}
  Let \(M = {\mathcal O}_{\lambda} \cong {\mathcal F}\ell ags(\C^3)\)
  be a full coadjoint orbit of
  the Lie group $PSU(3,\C)$ with maximal torus $T$ given by
  the standard diagonal subgroup. Here \(\lambda \in \t^* \subseteq
  \su(3)^*\) and ${\mathcal O}_{\lambda}$ is the $\lambda$-orbit of
  $PSU(3)$ with respect to the usual coadjoint action. 
  Equip $M={\mathcal O}_{\lambda}$
  with the Kostant-Kirillov-Souriau form $\omega_\lambda$ and let
  $\Phi: {\mathcal O}_{\lambda} \to \t^*$ be the $T$-moment map
  obtained by composing the projection \(\pi: \su(3,\C)^* \to \t^*\)
  with the inclusion \({\mathcal O}_{\lambda} \into \su(3,\C)^*.\) It
  is well-known that the $T$-action on $M$ is GKM, and that the equivariant $1$-skeleton of ${\mathcal
    O}_{\lambda}$ maps under $\Phi$ to the GKM graph pictured in grey in
  Figure~\ref{fig:flagsC3level}.

\begin{figure}[ht]
\centering
{
\psfrag{x}{$\xi$}
\psfrag{d}{$\beta^*$}
\psfrag{L}{$Lie(S^1)^*$}
\includegraphics[width=3in]{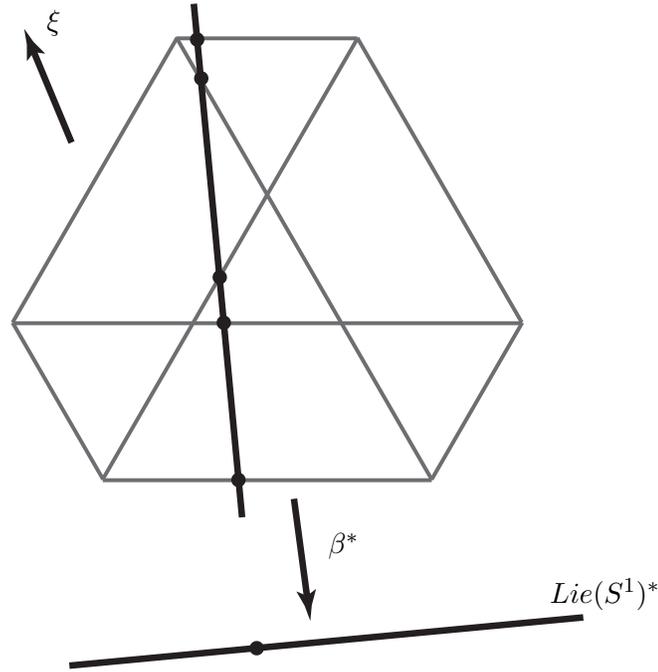}
}
\captionsetup{width=6in}
\caption{In grey, we indicate the image of the equivariant $1$-skeleton of $M$.  The $T$-fixed points correspond to the six (corner) vertices of the graph.  The black line intersecting the polytope represents the moment image of the level set $Z$ of an $S^1$-moment map $\Phi_{S^1}$. There
  are $5$ critical components $C_i$ in $\Crit(f)$, corresponding to
  the $5$ thick black dots (the images of the $C_i$ under $\Phi$).}
\label{fig:flagsC3level}
\end{figure}

For a choice of
$\beta: S^1 \into T$ such that ${\mathcal
  O}_{\lambda}^{S^1} = {\mathcal O}_{\lambda}^T$, the level set $Z$ of
the $S^1$-moment map 
$\Phi_{S^1} = 
\beta^* \circ \Phi$ is schematically indicated in Figure~\ref{fig:flagsC3level} by
the thick black line; the (images 
under $\Phi$
of the) components of $\Crit(f)$ for
a generic choice of $f = \Phi^\xi \vert_Z$ are
indicated by the thick black dots. 

The standard maximal-torus $T$-action on a coadjoint
orbit of a compact connected Lie group $G$ is GKM; hence we may apply
Theorem~\ref{theorem:GKM}. From Figure~\ref{fig:flagsC3level} we see
that, additively, $NK_{S^1}(Z)=\okf([Z/S^1])$ is a direct sum of representation rings
$R(\Gamma_{i,t})$, one for each critical component $C_{i,t}$ in
$\Crit(f \vert_{Z^t})$, as $t$ ranges in $S^1$. In fact, only
finitely many $t \in S^1$ will contribute nontrivial summands.
Here the subgroup $\Gamma_{i,t}$ of $S^1$ is the finite stabilizer group of
a point $p$ in $C_{i,t}$, which in turn may be computed in a
straightforward manner by analyzing the intersection of the chosen
$S^1$ with each of the stabilizer subgroups appearing in the
$T$-orbit stratification of ${\mathcal O}_{\lambda}$
(cf. \cite[Appendix B]{GGK}). 
\end{example}

\begin{example}\label{example:B2}
Now we consider the Lie type $B_2$. Here we find it convenient
to work with the 
complex form $PSO(5,\C)$. We recall that
the maximal torus $T$ of type $B_2$ is $2$-dimensional and the roots
are given as in Figure~\ref{fig:typeB2-roots}. 
We consider a coadjoint orbit $M = {\mathcal O}_{\lambda}$,
which may be identified with the homogeneous space $SO(5,\C)/P_{\alpha_1}$ where $P_{\alpha_1}$ is the
parabolic subgroup corresponding to the positive simple root
$\{\alpha_1\}$. More specifically, we may take 
${\mathcal O}_{\lambda}$ to be the 
coadjoint orbit through the element \(\lambda \in \t \cong
\t^*\) indicated in Figure~\ref{fig:typeB2-roots}.
The image of the equivariant $1$-skeleton for the Hamiltonian $T$-action on ${\mathcal O}_{\lambda}
\cong PSO(5,\C)/P_{\alpha_1}$ is depicted in
Figure~\ref{fig:typeB2-xray}.

\begin{figure}[ht]
\begin{minipage}[b]{0.45\linewidth}
\centering
{
\psfrag{A}{$\alpha_1$}
\psfrag{B}{$\alpha_2+2\alpha_1$}
\psfrag{C}{$\alpha_2+\alpha_1$}
\psfrag{D}{$\alpha_2$}
\psfrag{l}{$\lambda$}
\includegraphics[width=2in]{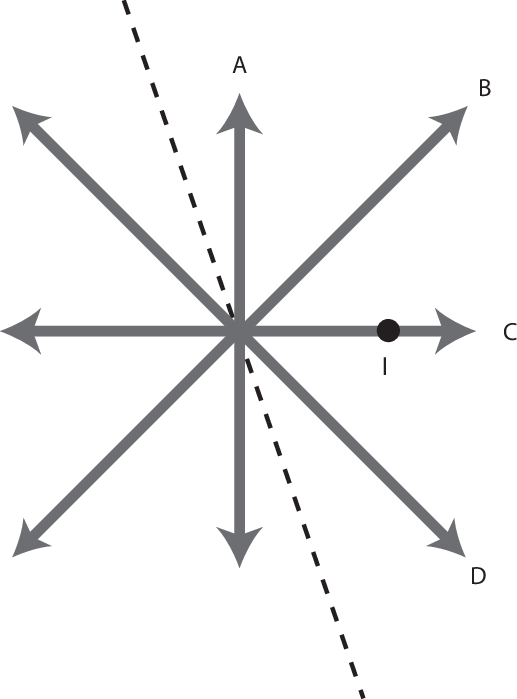}
}
\captionsetup{width=2.9in}
\caption{The type $B_2$ root system.
The dotted line is the hyperplane distinguishing 
  the positive from the negative roots. The element $\lambda$ lying on
the line spanned by $\alpha_1 + \alpha_2$ specifies the coadjoint orbit 
${\mathcal O}_{\lambda}$. }
\label{fig:typeB2-roots}
\end{minipage}
\hspace{0.5cm}
\begin{minipage}[b]{0.45\linewidth}
\centering
{
\psfrag{x}{$\xi$}
\psfrag{d}{$\beta^*$}
\psfrag{L}{$Lie(S^1)^*$}
\includegraphics[width=2.25in]{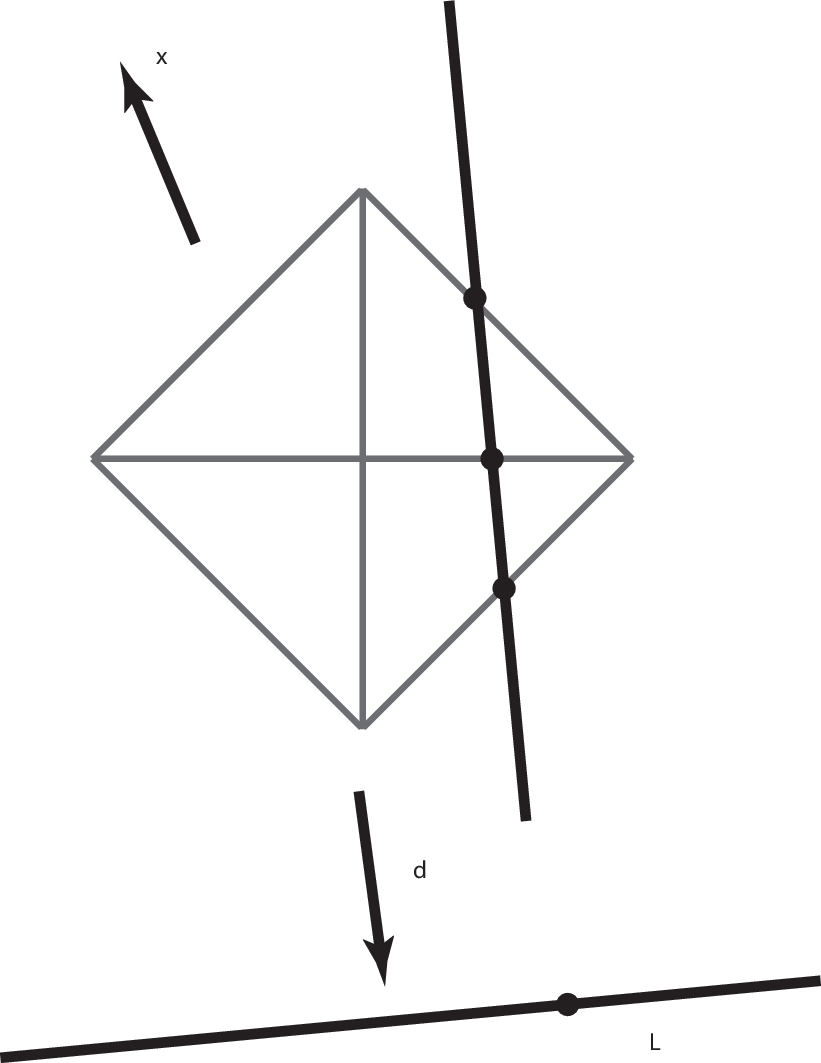}
}
\captionsetup{width=2.9in}
\caption{In grey, we indicate the image of the equivariant $1$-skeleton of $M = {\mathcal
    O}_{\lambda} \cong SO(5,\C)/P_{\alpha_1}$. The $T$-fixed points
  are the $4$ outer vertices. The black line intersecting the polytope represents the moment image of the level set $Z$ of an $S^1$-moment map $\Phi_{S^1}$. There
  are $3$ critical components $C_i$ in $\Crit(f)$, corresponding to
  the $3$ thick black dots (the images of the $C_i$ under $\Phi$).} 
\label{fig:typeB2-xray}
\end{minipage}
\end{figure}

Given $S^1\subset T$ with $M^{S^1}=M^T$ and corresponding moment map $\Phi_{S^1}$, the level set $\Phi_{S^1}^{-1}(\eta)$ indicated (under its image under $\Phi$) in the figure evidently lies entirely within an open Bruhat cell
of $M$. This Bruhat cell may be modelled on a single
linear $T$-representation with $T$-weights $-\alpha_2,
-\alpha_1-\alpha_2, -2\alpha_1-\alpha_2$, which renders the explicit
computation of the relevant finite stabilizer subgroups $\Gamma_{i,t}
\subseteq S^1$ particularly straightforward. This observation 
motivates the discussion in the next section. 

\end{example}

\section{Semilocally Delzant spaces}\label{sec:semilocal}

We have already seen in Sections~\ref{sec:toric} and~\ref{sec:GKM}
that the hypotheses of Corollary~\ref{cor:iso} are satisfied in several
situations familiar in equivariant symplectic geometry.  We will now
see that the methods of proof used thus far in this manuscript allow
us to make inductive use of the Main Theorem to cover more cases of 
orbifold symplectic quotients. Specifically, we observe that the 
proof of Theorem~\ref{thm:toric} shows that the $H$-equivariant
$K$-theory of the level set $Z$ arising from a Delzant construction
has the properties that $K^0_H(Z)$ is additive-torsion-free and
$K^1_H(Z)=0$. Therefore, for $(M,\omega,\Phi)$ a Hamiltonian $T$-space
and 
$\beta: H \into T$ a connected subtorus, if it can be shown that
each of the connected components of the critical sets appearing in
Theorem~\ref{thm:main} can be $H$-equivariantly identified with a
level set of a Delzant construction, then the hypotheses (3a) and (3b)
of Theorem~\ref{thm:main} would be satisfied, thus allowing us to apply
the Main Theorem to a wider class of symplectic quotients.

To this end, we make the following definition.

\begin{definition}\label{def:semilocal} 
  Let $(M, \omega, \Phi_H)$ be a Hamiltonian $H$-space with moment map
  $\Phi_H: M \to \h^*$. 
We will say that an $H$-invariant subset
  $C\subset M$ is {\bf semilocally Delzant
    with respect to $H$} if the following conditions are satisfied:
\begin{enumerate} 
\item  There exists a $2n$-dimensional $H$-invariant symplectic submanifold $N
  \subseteq M$, an $H$-invariant 
  open neighborhood $U \subseteq N$ of $C$, and a $H$-equivariant
  symplectomorphism 
\[
\psi: U \to V \subseteq \C^n
\]
for an open $H$-invariant subset $V \subseteq \C^n$, where $H$ acts
linearly on $\C^n$, with associated moment map $\Phi_{\C^n}: \C^n
\to \h^*$. 
\item Under the map \(\psi,\) the set $C$ is identified with a level set of
the induced $H$-moment map on $\C^n$. In other words, \(\psi(C) =
\Phi_{\C^n}^{-1}(\eta') \subseteq
  \C^n\) for some regular value $\eta' \in {\mathfrak h}^*$.
\item There exists $\xi \in \h$ such that $\Phi^\xi_{\C^n}
  \vert_{\psi(C)}$ is proper and bounded below. 
\end{enumerate}
\end{definition}

We take a moment to discuss
situations in equivariant symplectic geometry in which we may expect
the above definition to be applicable. 
Recall that the equivariant Darboux theorem states that, near an isolated
$H$-fixed point \(p \in M^H,\) there exists an open
neighborhood $U_p$ of $p$ which is $H$-equivariantly symplectomorphic
to an affine space $\C^n$ equipped with a linear $H$-action (here $p$
is identified with the origin $0$ of $\C^n$). Under some technical
assumptions (cf. \cite{GGK}) which are not very restrictive in practice, it is also possible to arrange 
the symplectomorphism such that the $H$-isotypic decomposition 
\[
\C^n \cong \bigoplus_{\alpha} \C_\alpha,
\]
where the sum is over weights $\alpha \in \h^*_{\Z}$ and $\C_\alpha$
denotes the subspace of $\C^n$ of weight $\alpha$, has the property
that the moment map $\Phi_{\C^n}$ associated to this $H$-action has a
component which is proper and bounded below. It is then evident that a
closed subset $C$ of $M$ which lies entirely inside such an
equivariant neighborhood $U_p \cong \C^n$ near $p \in M^T$, and which
may be identified with a level set of $\Phi_{\C^n}$ via the
equivariant Darboux theorem, is semilocally Delzant. Moreover, similar
statements could be made of subsets $C'$ of $M$ which lie entirely in
proper coordinate subspaces of $\C^n$ under the same equivariant
identification with $U_p \subseteq \C^n$. Informally, we may
say that $H$-invariant closed subsets which are ``near enough to an isolated fixed
point'' can be semilocally Delzant as described above. In particular,
this point of view leads
to concrete examples of symplectic quotient constructions 
(e.g. of Hamiltonian $H$-spaces with isolated fixed points, such as those where the $H$-action
is GKM) with critical sets $C$ satisfying
Definition~\ref{def:semilocal}. 

Indeed, we note that a concrete family of examples of Hamiltonian $T$-spaces with
well-known such equivariant neighborhoods are the flag varieties
(coadjoint orbits) $G/B$ and $G/P$ of compact connected Lie
groups. The maximal torus $T$ of the compact connected Lie group $G$
acts naturally on such homogeneous spaces, with fixed points corresponding to cosets
$W/W_P$. Moreover, $G/B$ (similarly $G/P$) has a
convenient open cover obtained by Weyl translates of the open Bruhat cell $Bw_0B/B$, where
$w_0$ is the longest word in the Weyl group.
By using \cite[Proposition~2.8]{KarTol} and some knowledge of the $T$-orbit stratification
of $G/B$, it is possible to identify a ``large" open subset $U$ of a Bruhat cell which provides such an 
equivariant Darboux neighborhood.  Moreover, the subset $U$ can be concretely described
in terms of moment map data.
If a closed subset $C$ of $G/B$ (similarly $G/P$)
may be seen to lie entirely within such a subset $U$, then the $T$-action near $C$ may be 
modelled by a linear
$T$-action on $\C^{\ell(w_0)}$ (here $\ell(w_0)$ denotes the Bruhat length
of $w_o$). We illustrate a concrete example of such a situation, using
a 
non-generic coadjoint orbit of Lie type $B_3$ in
Example~\ref{example:B3} below.

Returning to the discussion of orbifold $K$-theory, we first note that 
it is immediate from Theorem~\ref{thm:toric} that if $C$ is
semilocally Delzant, then $K^0_H(C)$ has no additive torsion and that
$K^1_H(C) = 0$. 
This leads to the following.

\begin{theorem}\label{theorem:semilocal}
Let $M$ be a Hamiltonian $T$ space, and let $H\subset T$ be a connected subtorus. Let $Z = \Phi^{-1}_H(\eta)\subset M$ be a level set of the $H$-moment map
$$
\Phi_H: M\rightarrow \mathfrak{h}^*
$$
and $\mathfrak{X} = [Z/H]$ be the orbifold obtained as a symplectic quotient $M/\!/H$.
Let $\xi\in \mathfrak{t}$ be such that $T^\xi = T$, and $f=\Phi^\xi
\vert_Z$ the corresponding moment map restricted to $Z$. 
Suppose that 
\begin{enumerate} 
\item $f$ is proper and bounded below,
\item for all \(t \in H,\) the set of
connected components $\pi_0(\Crit(f \vert_{Z^t}))$ is finite,
\item for all $t \in H$, each connected component $C$ of
 $\Crit(f|_{Z^t})$ is  semi-locally Delzant 
    with respect to $H$. 
\end{enumerate} 
Then $\okf({\mathfrak X})$ has no additive torsion.  Furthermore, $K^1_H(Z^t)=0$ for all $t\in H$.
\end{theorem}

\begin{proof}
 By Theorem~\ref{thm:toric}, $\okf([C/H])$ has no additive torsion
  for each connected component $C$ of $Crit(f|_{Z^t})$ for all $t \in
  H$. In particular, $K^0_H(C)$ has no torsion.  Since we also have
  $K_H^1(C)=0$ for all critical sets, we have satisfied the hypotheses
  of Theorem~\ref{thm:main}.  Hence 
 $\okf(\mathfrak{X})$ has no additive torsion.
\end{proof}

\begin{remark}\label{remark:Zsemilocal}
  We note that if the level set $Z$ itself is semilocally Delzant,
  then by transferring all analysis to the appropriate equivariant
  Darboux neighborhood \(U \subseteq \C^n\) and using the same
  argument as in Section~\ref{sec:toric}, we immediately see that for
  all \(t \in H,\) all connected components $C$ of $\Crit(f
  \vert_{Z^t})$ are semilocally Delzant with respect to $H$. Hence, in
  this case the hypothesis (3) above is automatically satisfied. 
\end{remark}

\begin{example}\label{example:B3}
We close with an example of a symplectic quotient of a type $B_3$
coadjoint orbit by a $2$-dimensional torus. Since the subtorus is
dimension $2$, Theorem~\ref{theorem:GKM} does not apply, but
we may
use Theorem~\ref{theorem:semilocal}. 
Recall that the complex form of the compact Lie group of type $B_3$ is
$PSO(7,\C)$. The maximal torus $T$ is $3$-dimensional, and the root
system is depicted in Figure~\ref{fig:typeB3-roots}. We denote the
associated moment map by $\Phi$.

\begin{figure}[ht]
\begin{minipage}[b]{0.45\linewidth}
\centering
{
\psfrag{1}{\small $\alpha_1=L_1-L_2$}
\psfrag{2}{\small $\alpha_2=L_3$}
\psfrag{3}{\small $\alpha_3=L_2-L_3$}
\psfrag{4}{\small $L_1=\alpha_1+\alpha_2+\alpha_3$}
\psfrag{5}{\small $L_2=\alpha_2+\alpha_3$}
\psfrag{6}{\small $L_1-L_3=\alpha_1+\alpha_3$}
\psfrag{7}{\small $L_1+L_2=$}
\psfrag{0}{\small $\alpha_1+2\alpha_2+2\alpha_3$}
\psfrag{8}{\small $L_1+L_3=\alpha_1+2\alpha_2+\alpha_3$}
\psfrag{9}{\small $L_2+L_3=2\alpha_2+\alpha_3$}
\includegraphics[width=2.25in]{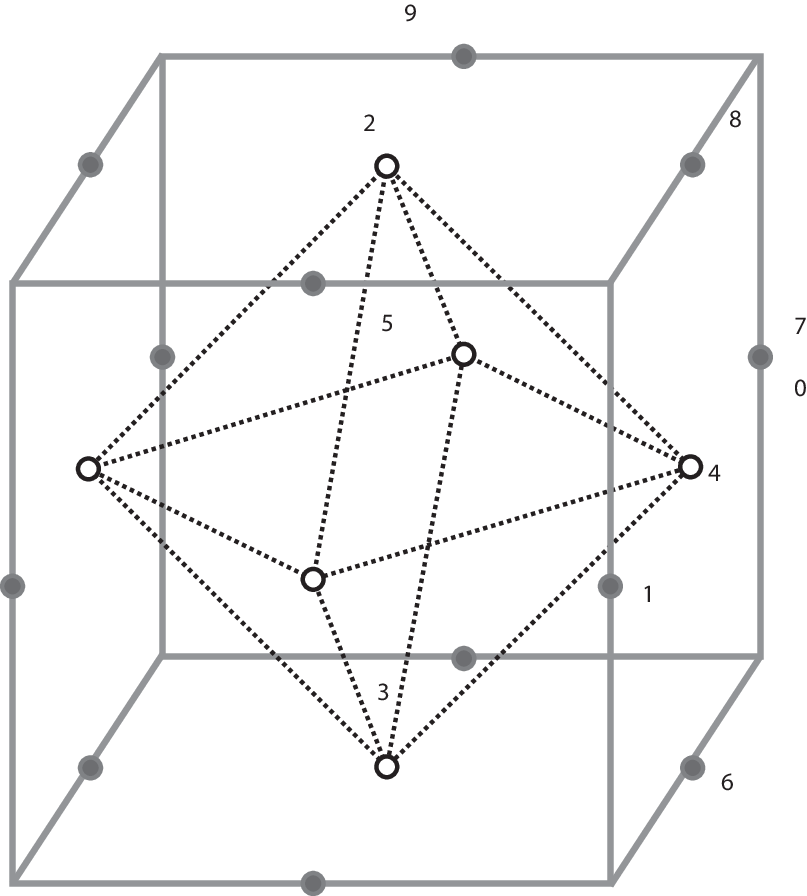}
}
\captionsetup{width=2.9in}
\caption{The root diagram for type $B_3$ with positive simple roots
  $\alpha_1, \alpha_2, \alpha_3$ (for details, see \cite[\S19.3]{FulHar91}). 
  The element \(\lambda \in \t^*\)
  lies on the positive span of the positive root $L_1 = \alpha_1 +
  \alpha_2 + \alpha_3$.} 
\label{fig:typeB3-roots}
\end{minipage}
\hspace{0.5cm}
\begin{minipage}[b]{0.45\linewidth}
\centering
{
\psfrag{A}{$p_1$}
\psfrag{B}{$p_2$}
\psfrag{C}{$p_3$}
\psfrag{p}{$\pi_{T'}$}
\psfrag{L}{$Lie(T')^*$}
\includegraphics[width=2.25in]{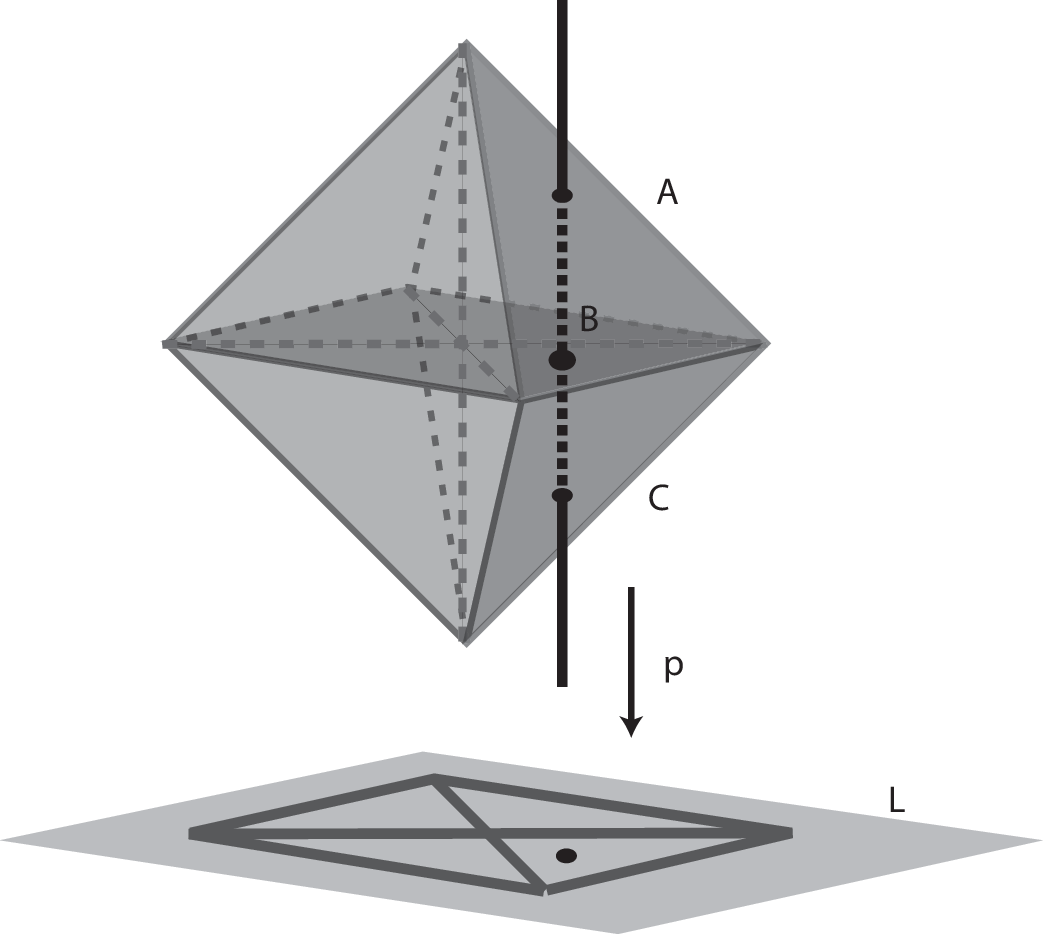}
}
\captionsetup{width=2.9in}
\caption{The GKM graph for $M = {\mathcal O}_{\lambda} \cong
  PSO(7,\C)/P_{\alpha_2, \alpha_3}$. The thick line and thick black
  dots schematically illustrate the (image under $\Phi$ of the)
  inverse images $Z := (\Phi_{T'})^{-1}(\eta)$ and the critical
  components of $\Crit(\Phi^\xi \vert_Z)$, respectively.}
\label{fig:typeB3-xray}
\end{minipage}
\end{figure}

We choose to work with a non-generic coadjoint orbit ${\mathcal
  O}_{\lambda}$ which may be identified with the 
complex homogeneous space $PSO(7,\C)/P_{\alpha_2, \alpha_3}$ where
$P_{\alpha_2, \alpha_3}$ is the parabolic
subgroup corresponding to the subset of the positive simple roots
$\{\alpha_2, \alpha_3\}$. 
We choose $\lambda$ lying on the positive span of the positive root $L_1 =
\alpha_1+\alpha_2+\alpha_3$ as in Figure~\ref{fig:typeB3-roots}.
The GKM graph of ${\mathcal O}_{\lambda}$
is also schematically shown. 
The image of the equivariant $1$-skeleton of
$M = {\mathcal O}_{\lambda}$ includes 
the three $2$-dimensional interior quadrilaterals given by the convex
hull of the roots $\{\pm L_1, \pm L_2\}, \{\pm L_2, \pm L_3\}, \{\pm
L_1, \pm L_3\}$ respectively. 

Let $T' \subset T$ be the $2$-dimensional connected subtorus of $T$ 
corresponding to the $2$-plane spanned by the roots $\{\pm L_1, \pm
L_2\}$, 
with corresponding projection
$\pi_{T'}: \t^* \to \Lie(T')^*$.  We wish to compute $\okf$ of the
symplectic quotient ${\mathcal O}_{\lambda} \mod T'$. 
The preimage \(\pi_{T'}^{-1}(\eta) \cap
\Delta\) in \(\Delta = \Phi({\mathcal O}_{\lambda})\) of a generic
regular value $\eta \in \Lie(T')^*$ is depicted as the
thick interval in Figure~\ref{fig:typeB3-xray}.

We wish now to show that 
the full orbifold $K$-theory of the quotient ${\mathcal O}_{\lambda}
\mod T'$ is free of additive torsion by
using Theorem~\ref{theorem:semilocal}. 
There are several ways to proceed. The first method, which depends on
Remark~\ref{remark:Zsemilocal}, is to observe that the full
level set $Z$ is semilocally Delzant. 
In this case, we may apply \cite[Proposition~2.8]{KarTol} to see that the
thick vertical line in Figure~\ref{fig:typeB3-xray} lies in an equivariant Darboux
neighborhood of the
$T$-fixed point $p$ corresponding to the root $L_1 = \alpha_1 +
\alpha_2 + \alpha_3$.  The $T$-action and corresponding moment map
$\Phi$ restricted to this neighborhood may be identified with that of a
linear $T$-action on $\C^5$ with weights $\{-L_1, -L_1 \pm L_2, -L_1
\pm L_3\}$ on the coordinates. The $T'$-action is the restriction of
this linear $T$-action, hence $Z$ is semilocally Delzant with respect
to $T'$. By Remark~\ref{remark:Zsemilocal} we may immediately apply
Theorem~\ref{theorem:semilocal}, as desired. 

In order to illustrate the concrete, straightforward nature of our
method of computation, for this example we 
also briefly
sketch the explicit analysis of each of the components of $\Crit(f)$ for
appropriate $f = \Phi^\xi \vert_Z$. Analysis of $\Crit(f
\vert_{Z^t})$, for $t \neq 1$,
would of course be similar. We begin by choosing $\xi$ generic such
that $\Crit(f)$ consists of the three components schematically
indicated in Figure~\ref{fig:typeB3-xray}. 

Observe that 
 the situations of the two exterior points $p_1, p_3$ in
 $\pi_{T'}^{-1}(\eta) \cap \Delta$
lying on the boundary $\partial \Delta$ are evidently symmetric,
 so it suffices to do the computations for only one of them. 
 We begin with the top exterior point $p_1$. 
A straightforward analysis of the linear $T$-action on the Bruhat cell
described above shows that $\Phi^{-1}(p_1) \subseteq {\mathcal
O}_{\lambda}$ consists of a single $T$-orbit diffeomorphic to $T'$. 
Moreover, the intersection of the stabilizer of the Bruhat cell with $T'$
is trivial, so $p_1$ corresponds to a free $T'$-orbit. Hence the
contribution to the full orbifold $K$-theory coming from $p_1$ is the
(ordinary) $K$-theory of a point, and is hence torsion-free.

We now proceed with the interior point $p_2$. 
(One way to view this computation is to recall that the horizontal quadrilateral
obtained as the convex hull of the roots $\{\pm L_1,
\pm L_2\}$ corresponds to a subvariety of
$PSO(7,\C)/P_{\alpha_2, \alpha_3}$ which may be identified with the homogeneous space of
$PSO(5,\C)$ of type $B_2$ studied in a previous example, although this
is not necessary for the computation.) 
Another straightforward analysis of linear $T$-actions, using the
explicit list of $T$-weights given above, 
yields that the corresponding symplectic quotient is the
``teardrop'' orbifold, 
i.e. the weighted projective space
$\P(1,2)$ (following notation of \cite{GolHarHolKim08}). Hence the
contribution to the full orbifold $K$-theory of ${\mathcal
  O}_{\lambda} \mod_\mu T'$ coming from the interior point $p_2$ is that associated
to $\P(1,2)$, which is explicitly computed in
\cite{GolHarHolKim08}, 
and has no additive torsion. 

\end{example}

\def\cprime{$'$}

\end{document}